\documentclass[12pt]{article}
\usepackage[a4paper]{geometry}
\usepackage{amsthm}

\usepackage{amssymb}
\usepackage{amsmath}
\usepackage{eufrak}

\newtheorem{theorem}{Theorem}

\newtheorem{corollary}[theorem]{Corollary}
\newtheorem{lemma}[theorem]{Lemma}

\newtheorem{definition}[theorem]{Definition}
\newtheorem{question}{Question}
\newtheorem{proposition}[theorem]{Proposition}

\begin{document}
\def\F{{\mathbb F}}
\title{ A new  formula for   Lazard's correspondence for finite  braces and pre-Lie algebras }

\author{Agata Smoktunowicz}
\date{}
\maketitle
\begin{abstract} In this paper a  simple algebraic formula is obtained for the correspondence between finite right nilpotent ${\mathbb F}_{p}$-braces and finite nilpotent pre-Lie algebras. This correspondence  agrees with the correspondence  using Lazard's correspondence   between   finite $\mathbb F_{p}$-braces and pre-Lie algebras  proposed  by Wolfgang  Rump in 2014.  
 As an application example, a classification  of  all 
 right nilpotent ${\mathbb F}_{p}$-braces generated by one element of cardinality $p^{4}$ is obtained. It is also  shown that the sum of  a finite number of left nilpotent ideals in a left brace is a left nilpotent ideal, therefore every finite brace contains the largest left nilpotent ideal.
 \end{abstract}

   The motivation  for this paper is the  following assertion, made by Wolfgang Rump on page 141 of \cite{Rump} for finite right braces: Suppose that $G$ is the adjoint group of a brace $A$. The $1$-cocycle $G\rightarrow A$ would then lead to a complete 
 Right RSA struture of $\bf g$ via Lazard's correspondence.

  We provide a formal proof of this correspondence as it appears none have been published.
 This correspondence means we can use pre-Lie algebras to characterise finite braces of cardinality $p^{n}$, and construct examples of braces using purely algebraic methods,  instead of the more typical group theory-based methods or  by computer or hand calculations.  This can  be  used to characterise the  structure of any finite brace, since it was shown in \cite{RN}  that every finite brace is completely determined by its adjoint group and braces which are its Sylow's subgroups.  
  
 In \cite{Rump}, page 135,  Rump developed a  connection between left nilpotent $\mathbb R$-braces and  pre-Lie algebras over the field of real numbers. In the case of pre-Lie algebras over finite fields, this method can be applied to obtain braces from pre-Lie algebras for sufficiently large $p$ using  Lazard's correspondence. However  it is not immediately  clear how to obtain a pre-Lie algebra from every  brace. It is also  not  clear if every brace will be an image of some pre-Lie algebra under Lazard's corresponence. 
 Therefore  it is not immediately clear how to attach 
  a pre-Lie algebra in a reversible way to every  brace, although  it is clear how to assign pre-Lie algebras to braces which were  already obtained from pre-Lie algebras using  Lazard's correspondence in the  different direction.

 Here we show how to attach to every finite  right nilpotent ${\mathbb F}_{p}$-brace a  finite nilpotent   pre-Lie algebra over the field ${\mathbb F}_{p}$. We develop a simple algebraic formula for this  passage from
 braces to pre-Lie algebras. For the  passage the other way,  from finite pre-Lie algebras to finite braces, we  can use the method from page 135 \cite{Rump}. We then show that this correspondence is one-to-one. Moreover, the passages from braces to pre-Lie algebras and from pre-Lie algebras to braces are reversible by each other. Therefore our formulas, which at first glance do not resemble Lazard's correspondence, in fact  correspond to Lazard's correspondence applied to  multiplicative groups of braces and therefore agree with the original suggestion by Rump. Notice that adjoint groups of braces are the multiplicative groups of braces under the operation $\circ $. 
  This provides us with easy-to-use formulas for  one-to-one correspondence between right nilpotent ${\mathbb F}_{p}$-braces and right nilpotent pre-Lie algebras over ${\mathbb F}_{p}$ of nilpotency index $k$ for $p>2^{k}$.

 As an example application we describe right  nilpotent  ${\mathbb F}_{p}$-braces of cardinality $p^{4}$ for $p>64$, which are generated by one element as braces. Some related results can be found in \cite{Cedo}.

 For information about Lazard's correspondence we refer the reader to \cite{K} and the references therein. The following questions  related to Lazard's correspondence and the  theory of braces remain unanswered: 
\begin{question}
 Let $A$ be an $\mathbb F_{p}$-brace of cardinality $p^{k}$ for some $k$. Does it follow that, for a sufficiently large $p$, $A$ with  the same additive operation $+$ and with operation $\cdot $ defined as 
\[a\cdot b=-\sum_{i=0}^{p-2}{\frac 1{2^{i}}}((2^{i}a)* b)\]
  is a pre-Lie algebra?
\end{question}
  \begin{question} Let $k$ be a natural number, and let $p$ be a prime number such that $p>2^{k}$. 
 Is there a bijective correspondence between  $\mathbb F_{p}$-braces of cardinality $p^{k}$
 and left nilpotent  pre-Lie algebras  over ${\mathbb F}_{p}$  of cardinality $p^{k}$ ? 
\end{question}

 The correspondence betwen groups with bijective $1$-cocycles and braces is mentioned in Theorem 2.1 in  \cite{Rio}.
  To construct braces from finite left nilpotent pre-Lie algebras following ideas from \cite{rump} it is possible to  use Lazard's correspondence for the corresponding Lie-algebra and then  construct a 1-cocycle on this group   $a\rightarrow p_{a}$ where $p_{a}(b)= e^{L_{a}}(b)$ (as on  page 135 in \cite{rump}).

  Notice that the relationship between braces, pre-Lie algebras, Rump's suggestions  and Lazard's correspondence  was subsequently  investigated in Section 3 \cite{DB}, where interesting related results have been obtained. Moreover in \cite{DB} the correspondence between braces and Hopf-Galois extensions of abelian type was discovered.
 Braces have found application in several research areas, some of which we now mention. They form an important concept in Hopf-Galois extensions --  see \cite{DB, LC, TC, gv,  kayvan, SVB} for related results.  They have been shown to be equivalent  to several concepts in group theory, such as groups with bijective 1-cocycles, regular subgroups of the holomorph, matched pairs of groups, factorised groups  and Garside Groups \cite{Rump, DB, Catino, FC, gateva, SVB, p, Sysak}. It is known that two-sided braces are exactly 
 the Jacobson radical rings \cite{rump}, and \cite{cjo, ILau}. In \cite{doikou}, applications of braces in quantum integrable
 systems were investigated, and in \cite{Agatka1} R-matrices constructed from braces were studied.
 Solutions  of  the pentagon equation  related to braces have  been investigated by several
 authors \cite{pent}.    In \cite{TB}, Brzezi{\' n}ski showed that braces are  related to trusses. Simple braces were investigated in \cite{DB, djc, cedo}, and  cohomology of braces was investigated in \cite{LV}. An analogon of the  Artin-Wedderburn theorem for braces was obtained in  \cite{JespersLeandro}. Subsequently skew braces have been introduced in \cite{gv}. 

  Many authors have developed methods to describe finite braces of a given cardinality.
 In particular, all braces and skew braces of cardinality $p^{3}$ for all prime numbers $p$ were described by David  Bachiller \cite{Ba}  and Kayvan Zenouz \cite{kayvan}. The description of all braces of cardinality $p^{4}$ is still unknown, however there is an estimate of a number of $\mathbb F_{p}$-braces of cardinality $p^{n}$   obtained by Lindsay Childs in \cite{L}. Some related open questions were posed in \cite{V, Dietzel}.

 In this paper we describe right nilpotent  one-generator braces of cardinality $p^{4}$. Notice that right nilpotent  braces generated by one element  are in correspondence with indecomposable  involutive non-degenerate set-theoretic solutions of a finite multipermutation level \cite{Agatka1, RumpEdi}, with indecomposable  solutions given by $\{y*x+x:y\in A\}$ where $x$ is a generator of $A$.

  In Section 4, the connection between braces and pre-Lie algebras is  used to generalise some results  from the context of  pre-Lie algebras to braces. For example, it is shown   that the sum of  a finite number of left nilpotent ideals in a left brace is a left nilpotent ideal. This is an analogon of the pre-Lie algebra result obtained in \cite{CKM}. Therefore every finite brace contains the largest left nilpotent ideal. 
 \section{Background information}

 Recall that a   {\em pre-Lie algebra} $A$ is a vector space with a binary operation $(x, y) \rightarrow  xy$
satisfying
\[(xy)z -x(yz) = (yx)z - y(xz),\]

 for every $x,y,z\in A$. We say that a pre-Lie algebra $A$  is {\em  nilpotent}  if, for some $n\in \mathbb N$, all products of $n$ elements in $A$ are zero. We say that $A$ is left nilpotent if for some $n$, we have $a_{1}\cdot (a_{2}\cdot( a_{3}\cdot (\cdots  a_{n})\cdots ))=0$ for all  $a_{1}, a_{2}, \ldots , a_{n}\in A$.
 Pre-Lie algebras were introduced by Gerstenhaber, and independently by Vinberg.

Recall that a set $A$ with binary operations $+$ and $* $ is a {\em  left brace} if $(A, +)$ is an abelian group and the following version of distributivity combined with associativity holds:
  \[(a+b+a*b)* c=a* c+b* c+a* (b* c), \space  a* (b+c)=a* b+a* c,\]
for all $a, b, c\in A$,  moreover  $(A, \circ )$ is a group, where we define $a\circ b=a+b+a* b$.

In what follows, we will use the definition in terms of operation `$\circ $' presented in \cite{cjo} (see \cite{rump} for the original definition): a set $A$ with binary operations of addition $+$, and multiplication $\circ $ is a brace if $(A, +)$ is an abelian group, $(A, \circ )$ is a group and for every $a,b,c\in A$
\[a\circ (b+c)+a=a\circ b+a\circ c.\]
 Circle algebras related to braces were introduced by Catino and Rizzo in
  \cite{Catino}.   
We now recall Definition $2$ from \cite{Rump}, which we state  for left braces, as  it was originally stated for right braces.
 Notice that 
$\mathbb F$-braces are related to circle algebras.
\begin{definition} Let $\mathbb F$ be a field. We say that a left  brace $A$ is an  $\mathbb F$-brace  
if its additive group is an $\mathbb F$-vector space such that 
$a*({\alpha }b)={\alpha }(a* b)$ for all $a,b\in A,$
$ {\alpha }\in \mathbb F$. Here $ a*b=a \circ b -a -b$.
\end{definition}

In \cite{rump}, Rump introduced {\em left nilpotent}  and  {\em right nilpotent}  braces and radical chains $A^{i+1}=A*A^{i}$ and $A^{(i+1)}=A^{(i)}*A$  for a left brace $A$, where  $A=A^{1}=A^{(1)}$. Recall that a left brace $A$  is left nilpotent if  there is a number $n$ such that $A^{n}=0$, where inductively $A^{i}$ consists of sums of elements $a*b$ with
$a\in A, b\in A^{i-1}$. A left brace $A$ is right nilpotent   if  there is a number $n$ such that $A^{(n)}=0$, where $A^{(i)}$ consists of sums of elements $a*b$ with
$a\in A^{(i-1)}, b\in A$.
 Strongly nilpotent braces and the chain of ideals $A^{[i]}$ of a brace $A$ were defined in \cite{Engel}.
 Define $A^{[1]}=A$ and $A^{[i+1]}=\sum_{j=1}^{i}A^{[j]}*A^{[i+1-j]}$.  A left brace $A$ is {\em strongly nilpotent}  if  there is a number $n$ such that $A^{[n]}=0$, where $A^{[i]}$ consists of sums of elements $a*b$ with
$a\in A^{[j]}, b\in A^{[i-j]}$ for all $0<j<i$.   Various other radicals in braces were subsequently introduced, in analogy with ring theory and group theory, see \cite{djc, ksv, kinneart, JespersLeandro}. In this paper we introduce left nilpotent radical for finite braces, in analogy with pre-Lie algebras.

\section{ From finite braces to Pre-Lie algebras}

Let $p>0$ be a prime number. Let ${\mathbb F}_{p}$ denote the field of cardinality $p$. Let $B$ be a left brace with operations $+$ and $\circ $, and operation $*$ where $a*b=a\circ b-a-b$.

Let $C_{p}$ be a cyclic group of order $p$. Let $B$ be a finite left brace whose additive group is ${C}_{p}\times \cdots \times { C}_{p}$, then $B$ is an  ${\mathbb F}_{p}$-brace.
 On the other hand, the additive group of an  ${\mathbb F}_{p}$ -brace is abelian, and every element has order $p$, so it is isomorphic to a group  ${C}_{p}\times \cdots \times {C}_{p}$, and hence $B$ has cardinality $p^{n}$ for some $n$.

By a result of Rump \cite{rump}, every brace of order $p^{n}$ is left nilpotent. Assume that $B$ is also right nilpotent, then by a result from \cite{Engel} it is strongly nilpotent. In other words there is $k$ such that the product of any $k$ elements, in any order,  is zero (where all products are under the operation $*$).  If $B^{[k]}=0$ and $B^{[k-1]}\neq 0$, then we will say that $B$ is  strongly nilpotent of degree $k$ (or we will say  that $k$ is the nilpotency index of $B$).

We recall  Lemma 15 from \cite{Engel}:

\begin{lemma}\label{fajny}
 Let $s$ be a natural number and let $(A, +, \circ)$ be a left brace such that $A^{s}=0$ for some $s$.
 Let $a, b\in A$, and as usual define $a*b=a\circ b-a-b$.
Define inductively elements $d_{i}=d_{i}(a,b), e_{i}=e_{i}(a, b)$  as follows:
$d_{0}=a$, $e_{0}=b$, and for $1\leq i$ define $d_{i+1}=d_{i}+e_{i}$ and $e_{i+1}=d_{i}*e_{i}$.
 Then for every $c\in A$ we have
\[(a+b)*c=a*c+b*c+\sum _{i=0}^{2s} (-1)^{i+1}((d_{i}*e_{i})*c-d_{i}*(e_{i}*c)).\]
\end{lemma}

{\bf Notation 1.} Let $A$ be a strongly nilpotent brace with operations $+, \circ , *$ defined as usual, so $x\circ y=x+y+x*y$
  for $x, y, z\in A$,  and 
let $E(x, y, z)\subseteq A$ denote the set consisting of any product of elements $x$ and $y$ and one element $z$ at the end of each product under the operation $*$,  in any order, with any distribution of brackets, each product consisting of at least 2 elements from the set $\{x,y\}$, each product having $x$ and $y$ appear at least once,  and having element $z$ at the end.  Moreover we only consider products of less than $k$ elements from the set $\{x,y,z\}$, where $k$ is the nilpotency index of $A$ (products of $k$ or more elements  are zero in $A$). Let $V_{x,y,z}$ be  a vector obtained from products of elements $x, y, z$ arranged in a such way that shorter products of elements  are situated  before longer products.

$ $

Below we associate to every strongly nilpotent  brace a pre-Lie algebra which is also strongly nilpotent and which has the same additive group.
 
\begin{proposition}\label{12345}
Let $A$ be an ${\mathbb F}_{p}$-brace which is strongly nilpotent of degree $k$.  Let $p$ be a prime number such that  
$2^{k}<p$.  As usual, the operations on $A$ are $+$, $\circ $  and $*$ where $a*b=a\circ b-a-b$. Define the binary operation $\cdot $ on $A$ as follows 
\[a\cdot b=\sum_{i=0}^{p-2}{\frac 1{2^{i}}}((2^{i}a)* b),\]
    for $a, b\in A$, where $2^{i}a$ denotes the sum of $2^{i}$ copies of element $a$, and $2^{-i}$ denotes the inverse of $2^{i}$ in ${\mathbb F}_{p}$.
Then $(a+b)\cdot c=a\cdot c+b\cdot c$ for every $a, b, c\in A$. 
Moreover $a\cdot (b+c)=a\cdot b+a\cdot c$ for every $a,b,c\in A$.
\end{proposition} 
\begin{proof} 
By the definition of a left ${\mathbb F}_{p}$-brace, we immediately get that  $a\cdot ( b+ c)=a\cdot b+ a\cdot c$. We will show that 
$(a+b)\cdot c=a\cdot c+b\cdot c$ for $a, b, c\in A$. 
 Observe that    \[(a+b)\cdot c= \sum_{i=0}^{p-2}{\frac 1{2^{i}}}((2^{i}a+2^{i}b)* c).\]
  Lemma \ref{fajny} applied several times yields 
\[{\frac 1{2^{n}}} (2^{n} a+{2^{n}}b)* c=
 {\frac 1{2^{n}}} (2^{n} a)*c+{\frac 1{2^{n}}}({2^{n}}b)*c+ {\frac 1{2^{n}}}C(n),\] 
 where $C(n)$ is a sum of some products of elements $2^{n}a$ and $2^{n}b$ and an element $c$ at the end
(because $A$ is a strongly nilpotent brace). Moreover, each product has at last one occurrence of element  ${2^{n}}a$ and also at last one occurrence of element  ${2^{n}}b$ and an element $c$ at the end.

 To show that  $(a+b)\cdot c =a\cdot c+b\cdot c$  it suffices to prove that \[\sum_{i=0}^{p-2}{\frac 1{2^{n}}}C(n)=0.\]
  We may consider a vector $V_{{\frac 1{2^{n}}}a, {\frac 1{2^{n}}}b, c}$ obtained as in Notation $1$ from products of elements ${\frac 1{2^{n}}}a$, ${\frac 1{2^{n}}}b$, $c$. 

By Lemma \ref{fajny} (applied several times) every element from the set $E(2x, 2y, z)$ can be written as a linear combination 
   of elements from $E(x, y, z)$, with coefficients which do not  depend on $x, y$ and $z$. We can assume that these coefficients belong to $\mathbb F_{p}$ since $pa=0$ for every $a\in A$    (we  use integers to represent elements of $\mathbb F_{p}$ in the modular arithmetic). We can then  organize these coefficients in a matrix, which we will call $M=\{m_{i,j}\}$,  so that we obtain
\[MV_{x,y, z}=V_{2x,2y, z}.\]

Notice that  elements from $E(x,y, z)$ (and from $E(2x, 2y, z)$) which are shorter appear before elements which are longer in our vectors $V_{x,y, z}$ and $V_{2x, 2y, z}$. Therefore by Lemma \ref{fajny} it follows that $M$ is an  upper triangular matrix.

Observe that the first four elements in the  vector  
 $V_{2x,2y, z}$ are \[(2x)*((2y)*z), ((2x)*(2y))*z, (2y)*((2x)*z), ((2y)*(2x))*z\] (arranged in some order). We can assume that $(2x)*((2y)*z)$ is the first entry in the vector $V_{2x, 2y, z}$ (so $x*(y*z)$ is the first entry in the vector $V_{x, y, z}$).  
 By Lemma \ref{fajny} applied several times,  $(2x)*((2y)*z)$ can be written as $4(x*(y*z))$ plus  elements  from $E(x,y,z)$ of degree larger than $3$ (so these elements are products of more than three elements from the set $\{x,y,z\}$). It follows that the first diagonal entry in $M$ equals $4$, so $m_{1, 1}=4$. 
Observe that the following diagonal entries can be calculated similarly, 
 for example, $(2x)*((2x)*((2x)*y))$ can be written using Lemma \ref{fajny} as $8(x*(x*(x*y)))$ plus  elements of degree larger than $4$.

Therefore $M$ is an upper triangular matrix with all diagonal entries of the form $2^{i}$, where $1<i< k$, where $k$ is the nilpotence index of our brace. It follows that all diagonal entries of the matrix  ${\frac 12}M-I$ are coprime with $p$  (since $2^{k}<p$), consequently 
  ${\frac 12}M-I$ is a non-singular matrix (where $I$ is the identity matrix), so it is an invertible matrix.

Notice that $M$ does not depend on $x, y$ and $z$, as we only used relations from Lemma \ref{fajny} to construct it. 
It follows that for every $n$, $M^{n}V_{x,y, z}=V_{2^{n}x,2^{n}y, z}$, therefore 
\[{\frac 1{2^{n}}}V_{2^{n}x, 2^{n}y, z}=({\frac 12}M)^{n}V_{x, y, z}.\]

Observe that $2^{p-1}x=x$ and $2^{p-1}y=y$ because  $2^{p-1}=1$ in ${\mathbb F}_{p}$.
 Therefore 
\[V_{x, y, z}={\frac 1{2^{p-1}}}V_{2^{p-1}x, 2^{p-1}y, z}=({\frac 12}M)^{p-1}V_{x, y, z}.\]

Notice that there is a vector $V$ with entries in ${\mathbb F}_{p}$ such that 
\[ C(n)=V^{T}V_{2^{n}x, 2^{n}y,z}=V^{T}M^{n}V_{x,y, z}\] for each $n$, where $V^{T}$ is the transposition of $V$. 

 Denote $M^{0}=I$ the identity matrix. Now we calculate

\[\sum_{n=0}^{p-2}{\frac 1{2^{n}}}C(n)=V^{T}(\sum_{n=0}^{p-2}{\frac 1{2^{n}}}M^{n}V_{x,y,z}).\]

Notice that $\sum_{n=0}^{p-2}{\frac 1{2^{n}}}M^{n}=(I-{\frac 12}M)^{-1}(I-({\frac 12}M)^{p-1}),$
 therefore 
\[\sum_{n=0}^{p-2}{\frac 1{2^{n}}}C(n)=V^{T}(I-{\frac 12}M)^{-1}((I-({\frac 12}M)^{p-1})V_{x,y,z})=0.\]
This concludes the proof. 
\end{proof} 

$ $
 \begin{corollary} \label{mai}
  Let notation be as in Proposition \ref{12345}, then $(\alpha a+\beta b)\cdot c=\alpha (a\cdot c)+\beta (b\cdot c)$ 
 for all $\alpha , \beta \in {\mathbb F}_{p}$.
\end{corollary}
\begin{proof}
Notice that $(na)\cdot b=n(a\cdot b)$ by Proposition \ref{12345}. Therefore, for $0\neq n\in {\mathbb F}_{p}$ and $d=na$ we have   
 $(d\cdot b)=n({\frac 1n}d)\cdot b.$ Therefore, $({\frac mn}a)\cdot b=m({\frac 1m}a)\cdot b={\frac mn}(a\cdot b)$. 
\end{proof}

 We will now prove the main result of this section.

\begin{theorem}\label{main}
  Let $A$ be an ${\mathbb F}_{p}$-brace which is strongly nilpotent of degree $k$. Assume that 
$2^{k}<p$.  As usual, the operations on $A$ are $+$, $\circ $  and $*$, where $a*b=a\circ b-a-b$. Define the binary operation $\cdot $ on $A$ as follows 
\[a\cdot b=\sum_{i=0}^{p-2}{\frac 1{2^{i}}}((2^{i}a)* b),\]
    for $a, b\in A$, where $2^{i}a$ denotes the sum of $2^{i}$ copies of element $a$, and $2^{-i}$ denotes the inverse of $2^{i}$ in ${\mathbb F}_{p}$.
Then \[(a\cdot b)\cdot c-a\cdot (b \cdot c)=(b\cdot a)\cdot c-(b\cdot a) \cdot c     \] for every $a, b, c\in A$. 
\end{theorem}
\begin{proof}  By Lemma \ref{fajny} applied several times we get 
\[(x+y)*z=x*z+y*z +x*(y*z)-(x*y)*z +d(x,y,z),\]
\[ (y+x)*z=x*z+y*z+y*(x*z)-(y*x)*z +d(y,x, z),\]
where $d(x, y, z)=E^{T}V_{x, y, z}$ for some vector $E$ with entries in ${\mathbb F}_{p}$ which does not depend of $x, y, z$, and where $V_{x, y, z}$ is as in Notation $1$ (moreover $d(x, y, z)$ is a combination of elements with at least 3 occurences of elements from the set $\{x, y\}$). 
 It follows that 
\[x*(y*z)-(x*y)*z-y*(x*z)+(y*x)*z=d(y, x, z)-d(x, y, z).\]

Let $a,b,c\in A$ and let $i, j$ be natural numbers.
  Applying it to $x=2^{i}a$, $y=2^{j}b$, $z=c$ we get 
\[(2^{i}a)*((2^{j}b)*c)-((2^{i}a)*(2^{j}b))*c+d(2^{i}a, 2^{j}b,c)=\]
\[ =(2^{j}b)*((2^{i}a)*c)-((2^{j}b)*(2^{i}a))*c+d(2^{j}b,  2^{i}a, c).\]

Notice that  
\[a\cdot (b\cdot c)=a\cdot \sum_{j=0}^{p-2}{\frac 1{2^{j}}}((2^{j}b)* c)=\sum_{i=0}^{p-2}{\frac 1{2^{i}}}((2^{i}a)* 
\sum_{j=0}^{p-2}{\frac 1{2^{j}}}((2^{j}b)* c)).\]

 Consequently, 
\[a\cdot (b\cdot c)=\sum_{i, j=0}^{p-2}{\frac 1{2^{i+j}}}((2^{i}a)*((2^{j}b)* c)).\]

On the other hand 

\[(a\cdot b)\cdot c=(\sum_{i=0}^{p-2}{\frac 1{2^{i}}}((2^{i}a)* b))\cdot c=\sum_{i=0}^{p-2}{\frac 1{2^{i}}}(((2^{i}a)* b)\cdot c),\]
 where the last equation follows from Proposition \ref{12345}. 

Consequently, 

\[(a\cdot b)\cdot c=\sum_{i,j=0}^{p-2}{\frac 1{2^{i+j}}}((2^{j}((2^{i}a)* b))* c)=\sum_{i,j=0}^{p-2}{\frac 1{2^{i+j}}}((2^{i}a)* (2^{j}b))* c,\]

 Recall a previous equation, multiplied by ${\frac 1{2^{i+j}}}$ on both sides: 

\[{\frac 1{2^{i+j}}}(2^{i}a)*((2^{j}b)*c)-{\frac 1{2^{i+j}}}((2^{i}a)*(2^{j}b))*c+{\frac 1{2^{i+j}}}d(2^{i}a, 2^{j}b,c)=\]
\[ ={\frac 1{2^{i+j}}}(2^{j}b)*((2^{i}a)*c)-{\frac 1{2^{i+j}}}((2^{j}b)*(2^{i}a))*c+{\frac 1{2^{i+j}}}d(2^{j}b,  2^{i}a, c).\]

 By summing the above equation for all $0\leq i,j\leq p-2$ and subtracting the previous equations we obtain that

\[(a\cdot b)\cdot c-a\cdot (b\cdot c)+\sum_{i,j=0}^{p-2}{\frac 1{2^{i+j}}}d(2^{i}a, 2^{j}b,c)=\]
\[(b\cdot a)\cdot c-b\cdot (a\cdot c)+\sum_{i,j=0}^{p-2}{\frac 1{2^{i+j}}}d( 2^{j}b,2^{i}a, c).\]

 So it remains to show that \[\sum_{i,j=0}^{p-2}{\frac 1{2^{i+j}}}d(2^{i}a, 2^{j}b,c)=0,\] for all $a, b, c\in A$ 
 (and hence $\sum_{i,j=0}^{p-2}{\frac 1{2^{i+j}}}d( 2^{j}b,2^{i}a, c)=0$).

{\em Proof that  $\sum_{i,j=0}^{p-2}{\frac 1{2^{i+j}}}d(2^{i}a, 2^{j}b,c)=0$}. The proof  uses a similar idea as the proof of Proposition \ref{12345}, but we include all the details here for convenience. 

 Notice that \[d(a,b,c)=w(a,b,c)+v(a,b,c)\] where $w(a,b,c)$ contains all the products of elements $a,b,c$ which appear as summands in  $d(a,b,c)$ and in which $a$ appears at least twice, and 
$v(a,b,c)$ is a sum of products which are summands in $d(a,b,c)$ and in which $a$ appears only once (and hence $b$ appears at least twice).
It suffices to show that \[\sum_{i,j=0}^{p-2}{\frac 1{2^{i+j}}}w(2^{i}a, 2^{j}b,c)=0\] and 
\[\sum_{i,j=0}^{p-2}{\frac 1{2^{i+j}}}v(2^{i}a, 2^{j}b,c)=0.\]
 Observe that  it suffices to show that $\sum_{i=0}^{p-2}{\frac 1{2^{i}}}w(2^{i}a, b',c)=0$
   and  $\sum_{j=0}^{p-2}{\frac 1{2^{j}}}v(a', 2^{j}b, c)=0$ for any $a, a',  b, b', c\in A$ (notice that $a', b'$ denote arbitrary elements of $A$, and should not be confused with inverses of elements  $a$ and $b$).

 We will first show that \[\sum_{i=0}^{p-2}{\frac 1{2^{i}}}w(2^{i}a, b',c)=0.\]

Observe that there is a vector $W$ with entries in ${\mathbb F}_{p}$ such that $w(a,b',c)=W^{T}V_{a,b',c}'$ where $V_{a,b',c}'$ is a vector constructed as in Notation 1 but only including as entries these products from $E(a,b',c)$  in which $a$ appears at least twice. 
  By using Lemma \ref{fajny} several times (similarly as in the proof of Proposition \ref{12345}) there exists a matrix $M$ such that 
$V_{2a, b', c}'=MV_{a,b',c}'$ and $M$ is upper triangular with diagonal entries $2^{i}$ for $i\geq 2$ (because $a$ appears at least twice in each product which is an entry in $V_{a,b',c}'$)  and for $i< k$ because $k$ is the nilpotency index of $A$ (recall that $2^{k}< p$). Therefore ${\frac 12}M-I$ is an invertible matrix.
 We can assume that the entries of $M$  belong to $\mathbb F_{p}$ since $pa=0$ for every $a\in A$    (we  use integers to represent elements of $\mathbb F_{p}$). 

Therefore, \[ \sum_{i=0}^{p-2}{\frac 1{2^{i}}}w(2^{i}a, b',c)=\sum_{i=0}^{p-2}{\frac 1{2^{i}}}W^{T}V_{2^{i}a,b',c}'=\sum_{i=0}^{p-2}W^{T}{\frac 1{2^{i}}}M^{i}V_{a,b',c}'.\]

Notice that $2^{p-1}=1$ in ${\mathbb F}_{p}$, therefore \[V_{a,b',c}'=V_{2^{p-1}a,b',c}'=M^{p-1}V_{a,b',c}'.\]
 Notice that ${\frac 1{2^{p-1}}}=1$ in  ${\mathbb F}_{p}$. It follows that \[\sum_{i=1}^{p-2}{\frac 1{2^{i}}}M^{i}V_{a,b',c}'=(I-{\frac 12}M)^{-1}(({\frac 12})^{p-1}M^{p-1}-I)V_{a,b',c}'=\]
\[=(I-{\frac 12}M)^{-1}(M^{p-1}-I)V'_{a,b',c}=0.\]
 It follows that \[ \sum_{i=0}^{p-2}{\frac 1{2^{i}}}w(2^{i}a, b',c)=0.\]
 The proof that 
 $\sum_{j=0}^{p-2}{\frac 1{2^{j}}}v(a', 2^{j}b,c)=0$ for all $a', b,c\in A$ is similar.
 Observe that there is a vector $W'$ with entries in ${\mathbb F}_{p}$ such that $v(a',b,c)=W'^{T}V_{a', b,c}''$, where $V_{a', b, c}''$ is a vector constructed as in Notation 1 but only including as entries those products in which $b$ appears at least twice. 
 By applying Lemma \ref{fajny} several times,  there exists a matrix $\bar M$  such that 
$V_{a', 2b,  c}''={\bar M}V_{a', b,c}''$ and ${\bar M}$ is upper triangular with diagonal entries $2^{i}$ for $i\geq 2$ (because $b$ appears at least twice  in each product which is an entry in $V_{a',b,c}''$)  and for $i< k$. We can assume that the entries of $\bar M$ belong to $\mathbb F_{p}$ since $pa=0$ for every $a\in A$. 
 Similarly as before, \[V_{a',b,c}''=V_{a', 2^{p-1}b, c}''={\bar M}^{p-1}V_{a', b,c}''={\frac 1{2^{p-1}}}{\bar M}^{p-1}V_{a', b,c}''.\] It follows that \[ \sum_{i=0}^{p-2}{\frac 1{2^{j}}}v(a', 2^{j}b,c)=\sum_{i=0}^{p-2}{\frac 1{2^{j}}}W'^{T}V_{a',2^{j}b,c}''=\sum_{i=0}^{p-2}W'^{T}{\frac 1{2^{j}}}{\bar M}^{j}V_{a',b,c}''\]
hence \[\sum_{j=0}^{p-2}{\frac 1{2^{j}}}v(a', 2^{j}b, c)=(I-{\frac 12}{\bar M})^{-1}({\frac 1{2^{p-1}}}{\bar M}^{p-1}-I)V_{a',b,c}''=0.\]

\end{proof} 

 We obtain the following corollary: 
\begin{corollary}\label{890}

  Let $A$ be an ${\mathbb F}_{p}$-brace of degree $k$ which is strongly nilpotent. Assume that 
$2^{k}<p$.  As usual the operations on $A$ are $+$, $\circ $  and $*$ where $a*b=a\circ b-a-b$. Define the binary operation $\cdot $ on $A$ as follows 
\[a\cdot b=\sum_{i=0}^{p-2}{\frac 1{2^{i}}}((2^{i}a)* b),\]
    for $a, b\in A$ where $2^{i}a$ denotes the sum of $2^{i}$ copies of element $a$, and $2^{-1}$ denotes the inverse of $2^{i}$ in ${\mathbb F}_{p}$.
Define $a\odot b=-(a\cdot b)$, then $A$ with operations $+$ and $\odot $ is a pre- Lie algebra over the field ${\mathbb F}_{p}$.
\end{corollary}
\begin{proof} By Proposition \ref{12345}, $(a\odot b)\odot c=(-(a\cdot b)\odot c)=-(-(a\cdot b)\cdot c)=(a\cdot b)\cdot c$, similarly,
$a\odot (b\odot c)=-(a\cdot(-(b\cdot c))=a\cdot(b\cdot c)$. by Theorem \ref{main}, $A$ with operations $+, \odot $ is a pre-Lie algebra.
\end{proof}
 
\section{From pre-Lie algebras to braces}\label{mi} 

  Observe that if the pre-Lie algebra is nilpotent  we can use the group of flows of a pre-Lie algebra to obtain the passage from 
  finite nilpotent pre-Lie algebras of cardinality $p^{n}$ and right nilpotent ${\mathbb F}_{p}$-braces in a manner similar to \cite{new}.
 Upon closer inspection this gives the same brace when we use Lazard's correspondence and later change the obtained group with 1-cocycle into brace suggested by Rump in \cite{rump} on  pages 135, 141.
 The correspondence betwen groups with bijective $1$-cocycles and braces is mentioned in Theorem 2.1 in  \cite{Rio}.
 As mentioned by Rump in a private correspondence, the addition in the pre-Lie algebra and in the corresponding brace is always the same, so we only need to define the multiplication $\circ $ in the brace.

 Let $A$ with operations $+$ and $\cdot $ be a nilpotent pre-Lie algebra (over field ${\mathbb F}_{p}$) of nilpotency index $k$. Recall  that  a pre-Lie algebra $A$ is nilpotent of nilpotency index $n$ if  the product of any $n$ elements is zero in this pre-Lie algebra, and $n$ is minimal possible. 
 Let $p$ be a prime number larger than ${k}$ and let ${\mathbb F}_{p}$ be the field of $p$ elements. Define the $\mathbb F_{p}$-brace $(A,+, \circ )$ with the same addition as in the pre-Lie algebra $A$ and with the multiplication $\circ $ defined  as in the group of flows as follows. 
\begin{enumerate}

\item Let $a\in A$, and let  $L_{a}:A\rightarrow A$ denote the left multiplication by $a$, so
$L_{a}(b)=a\cdot b$.
 Define $L_{c}\cdot L_{b}(a)=L_{c}(L_{b}(a))=c\cdot (b\cdot a)$.
 Define \[e^{L_{a}}(b)=b+a\cdot b+{\frac 1{2!}}a\cdot (a\cdot b)+{\frac 1{3!}}a\cdot (a\cdot (a\cdot b))+\cdots \]
 where the sum `stops' at place $k$, since the nilpotency index of $A$ is $k$. This is well defined since $p>{k}$.

\item  We can formally consider element $1$ such that $1\cdot a=a\cdot 1=a$ in our pre-Lie algebra (as in \cite{M})  and
 define \[W(a)=e^{L_{a}}(1)-1=a+{\frac 1{2!}}a\cdot a+{\frac 1{3!}}a\cdot (a\cdot a)+ \cdots \]
 where the sum `stops' at  place $k$. 
 Notice that $W:A\rightarrow A$ is a bijective function, provided that $A$ is a nilpotent pre-Lie algebra.

\item Let $\Omega :A\rightarrow A$ be the inverse function to the function $W$, so  $\Omega (W(a))=W(\Omega (a))=a$.
  Following \cite{M} the first terms of $\Omega $ are
\[ \Omega (a)=  a-{\frac 12}a\cdot a +{\frac 14} (a\cdot a)\cdot a +{\frac {1}{12}}a\cdot (a\cdot a) +\ldots \]
 where the sum stops at place $k$. 
In \cite{M} the formula for $\Omega $ is given using Bernoulli numbers. This assures that $p$ does not appear in a denominator.
\item Define\[a\circ b=a+e^{L_{\Omega (a)}}(b).\]
 Here the addition is the same as in the pre-Lie algebra $A$.
It was shown in \cite{AG}  that $(A, \circ )$ is a group.  The same argument will work in our case, as $(W(a)\circ W(b))\circ W(c)=W(a)\circ (W(b)\circ W(c))$ for $a,b,c\in A$ by BCH formula (at this stage the result is related to Lazard's correspondence).  We can immediately  see that $(A, +, \circ )$ is a left brace because
 \[a\circ (b+c)+a=a+e^{L_{\Omega (a)}}(b+c)+a=(a+e^{L_{\Omega (a)}}(b))+(a+e^{L_{\Omega (a)}}(c))=a\circ b+a\circ c.\]

\end{enumerate}

The following question  remains unanswered: 
 
\begin{question}
 If $A$ is a pre-Lie algebra over a field $F_{p}$, is the map $W(a)$ a bijection, provided that $A$ is left nilpotent (but not necessarily  right nilpotent)? 
\end{question}

 Notice that the map  $W$ is well defined but it is not clear if it is a bijection. If  the answer is yes, then  it would be possible to get much stronger results than here, and to generalise the correspondence with pre-Lie algebras to all $\mathbb F_{p}$-braces of cardinality $p^{n}$.

\section{The correspondence is one-to-one}\label{fasola}

$ $

 In this chapter we show that the correspondence between strongly nilpotent $\mathbb F_{p}$-braces of nilpotency index $k$ and nilpotent pre-Lie algebras over $\mathbb F_{p}$ of nilpotency index $k$ is one-to-one for $p>2^{k}$, where  $p$ is a prime number. We begin with the following proposition.

\begin{proposition} \label{lim}
 Let $(A, +, \cdot )$ be a nilpotent  pre-Lie algebra over a field $\mathbb F_{p}$ of nilpotency index $k$, where $2^{k}<p$. Let $(A, +, \circ )$ be the  brace obtained as in Section \ref{mi}, so $(A, \circ )$ is the formal group of flows of the pre-Lie algebra $A$. Then $A$ has the nilpotency index $k$. Moreover, the following holds,
\[a\cdot b=-\sum_{i=0}^{p-2}{\frac 1{2^{i}}}((2^{i}a)* b),\]
    for $a, b\in A$, where $2^{i}a$ denotes the sum of $2^{i}$ copies of element $a$, and $2^{-i}={\frac 1{2^{i}}}$ denotes the inverse of $2^{i}$ in ${\mathbb F}_{p}$.
\end{proposition}
\begin{proof} Notice that, by using formulas from Section \ref{mi}, we see that since $(A, \circ )$ is the group of flows of the pre-Lie algebra $A$, then  \[a* b=a\cdot b +\sum _{w\in P_{a,b}} \alpha _{w}w\] where $\alpha _{w}\in {\mathbb F}_{p}$  and $P_{a,b}$ is the set of all products of elements $a$ and $b$ from $(A, \cdot )$ with $b$ appearing only at the end, and $a$ appearing at least two times in each product. Moreover, $\alpha _{w}$ does not depend on $a$ and $b$, but only on their arrangement in word $w$ as an element of  set $P_{a,b}$.
 This follows from the construction of $\Omega (a)$, which is a sum of $a$ and a  linear combination of all possible products of more than one element $a$ with any distribution of brackets, which can be proved by induction.
Notice that each word $w$ will be a product of at most $k$ elements because pre-Lie algebra $A$ has nilpotency index $k$.
 Let $w\in P_{a,b}$, then $w$ is a product of some elements $a$ and element $b$. We define the word $w_{2^{i}}$ to be the  word obtained if at each place where $a$ appears in $w$ we write $2^{i}a$ instead of a.
 It follows that:
 \[(2^{i}a)* b=(2^{i}a)\cdot b +\sum _{w\in P_{a, b}} \alpha _{w}w_{2^{i}}.\] 
 Consequently,
 \[\sum_{i=0}^{p-2}{\frac 1{2^{i}}}((2^{i}a)* b)=\sum_{i=0}^{p-2}{\frac 1{2^{i}}}[(2^{i}a)\cdot b+\sum _{w\in P_{a,b}} \alpha _{w}w_{2^{i}}].\]

Notice that \[\sum_{i=0}^{p-2}{\frac 1{2^{i}}}(2^{i}a)\cdot b=(p-1)a\cdot b=-a\cdot b.\]
 Therefore, it suffices to show that for every $w\in P_{a,b}$ we have \[\sum_{i=0}^{p-2}{\frac 1{2^{i}}}w_{2^{i}}=0.\] We know that every pre-Lie algebra is distributive, hence $w_{2^{i}}=(2^{i})^{j}w$ where $j$ is the number of occurences of $a$ in the product which gives $w$. 
 It suffices to show that $\sum_{i=0}^{p-2}{\frac 1{2^{i}}}(2^{ij})=0$ in ${\mathbb F}_{p}$. Because $2\leq j<k$ this is true in ${\mathbb F}_{p}$, as $\sum_{i=0}^{p-2}{\frac 1{2^{i}}}{2^{ji}}= ((2^{j-1})^{p-1}-1)(2^{j-1}-1)^{-1}=0$, which concludes the proof.

  Notice that by the formula for the multiplication $*$ in the group of flows, the nilpotency index in the constructed brace (as the group of flows) will be at most the same as the nilpotency index of the pre-Lie algebra $A$. On the other hand, the formula from the beginning of this proof assures that the nilpotency index of the pre-Lie algebra $A$ does not exceed the nilpotency index of the brace which is its group of flows. So the nilpotency indexes agree. 
\end{proof}

We now show that the correspondence is one-to-one:

\begin{theorem}\label{5} Let $p$ be a prime number and ${\mathbb F}_{p}$ be the field consisting of $p$ elements. 
 Let $(A, +, \circ )$ be a strongly nilpotent brace of nilpotency index $k$ where $2^{k}<p$. Let $(A, +, \cdot )$ be a nilpotent  pre-Lie algebra over the field $\mathbb F_{p}$ 
 obtained from this brace using Theorem \ref{main} and Corollary \ref{890}, so $(A, +, \cdot )$ has the same addition as brace $(A, +, \circ )$ and the multiplication is defined as
 \[a\cdot b=-\sum_{i=0}^{p-2}{\frac 1{2^{i}}} ({2^{i}}a)*b.\]
 
 Then $(A, \circ)$ is the group of flows of the pre-Lie algebra $(A, +, \cdot )$ and  $(A, +,  \circ )$ can be obtained as in 
 Section \ref{mi} from pre-Lie algebra $(A, +, \cdot )$. 
\end{theorem}
\begin{proof} Let $E_{a,b}\subseteq A$  denote the set consisting of any product of elements $a$ and one element $b$ at the end of each product under the operation $*$,  in any order, with any distribution of brackets, each product consisting of at least 2 elements $a$.
 Observe that by Lemma \ref{fajny} applied several times 
\[a\cdot  b=a* b +\sum_{w\in E_{a,b}}\alpha _{w}w\] where 
$\alpha _{w}\in F_{p}$ do not depend on $a,b$, but only on their arrangement in word $w$ as an element of $E_{a,b}$. Moreover, 
 each $w$ is a product of at least $3$ elements from the set $\{a, b\}$. Observe that coefficients $\alpha _{w}$ do not  depend on the brace $A$ as they were constructed using the formula from Lemma \ref{fajny} which holds in every strongly nilpotent brace. Therefore,  for any given $k$, the same formula will hold for all braces of nilpotency index $k$ on the set $A$.

Notice that the formula $a\cdot  b=a* b +\sum_{w\in E_{a,b}}\alpha _{w}w$ implies that 
 any product of $i$ elements in the pre-Lie algebra $A$ (under the operation $\cdot $) will belong to $A^{[i]}$. 
 Therefore we can use the formula  $a*b=a\cdot b-\sum_{w\in E_{a,b}}\alpha _{w}w$ several times  to write every element from $E_{a,b}$ as a product of elements $a$ and $b$ under the operation $\cdot $.
 In this way we can recover the brace $(A,  +, \circ )$ from  the pre-Lie algebra $(A, +, \cdot )$. 

 Notice that the nilpotency index of the pre-Lie algebra $(A, +, \cdot)$ will be the same as the nilpotency idex of the brace $(A, +, \circ )$. 
Therefore two distinct strongly nilpotent braces cannot give the same pre-Lie algebra using the formula $a\cdot b=-\sum_{i=0}^{p-2}{\frac 1{2^{i}}} ({2^{i}}a)*b$. 

Notice that, because we know that pre-Lie algebra $(A, +, \cdot )$ can be obtained as in Theorem \ref{main} from the brace which is it's group of flows (by Proposition \ref{lim}), it follows that 
 $(A,\circ)$ is the group of flows of pre-Lie algebra $A$. 
\end{proof}

 Therefore we obtain the following corollary.

\begin{corollary}\cite{L} Let $k$ be a natural number, and $p$ be a prime number such that $p>2^{k}$. Then
 there is one-to-one correspondence between the set of strongly nilpotent ${\mathbb F}_{p}$-braces of nilpotency index $k$ and the set  of  nilpotent pre-Lie algebras over ${\mathbb F}_{p}$ of nilpotency index $k$.
\end{corollary}
\begin{proof} 
 For every pre-Lie algebra of nilpotency index $k$ we can attach the brace which is its group of flows and form a   pair.  Since the group of flows is uniquely defined,  every pre-Lie algebra will be in exactly one pair.  Moreover, every brace will be in some pair, by Theorem \ref{5}.
 Observe that every brace will be in exactly one pair, as otherwise there would be two distinct pre-Lie algebras which give the same group of flows. However, by Proposition  \ref{lim} we can apply the formula 
\[a\cdot b=-\sum_{i=0}^{p-2}{\frac 1{2^{i}}}((2^{i}a)* b),\] to recover these pre-Lie algebras from this brace. Because the formula defines uniquely the underlying pre-Lie algebra  every brace is in at most one pair.   
  
\end{proof}

{\em Remark regarding connections with the BCH formula and with Lazard's correspondence: }  Let $(A, +, \cdot )$ be a  finite nilpotent pre-Lie algebra, and $(A, \circ )$ be its group of flows. Notice that the  formula for the multiplication in the group of flows  can also be written using the Baker-Campbell-Hausdorff formula and  Lazard's correspondence:
\[ W(a)\circ W(b)= W(C(a, b)),\]
 where $C(a, b)$ is obtained using the Baker-Campbell-Hausdorff  series in the Lie algebra $L(A)$ \cite{AG}, \cite{M}.  Therefore   there is a group isomorphism between the group $G(L(A))$ obtained from $L(A)$ by using the Baker-Campbell-Hausdorff  (BCH) formula and the group of flows $(A, \circ )$ of a pre-Lie algebra $A$ given by the map $p: a\rightarrow W(a)$.
It follows from the formula  $W(a)\circ W(b)= W(C(a, b))$.   Recall that the Lie algebra $L(A)$ is obtained from the  pre-Lie algebra $A$ by taking $[a,b]=a\cdot b-b\cdot a$, and has the same addition as  $A$.

\begin{itemize} 
 \item Applying the inverse of the BCH formula  to the group $G(L(a))$ gives the map $G(L(A))\rightarrow L(A)$.
 Notice that when we apply the inverse of the BCH formula to  the group of flows $(A, \circ )$ we are applying the inverse of the BCH formula to the group isomorphic to $G(L(A))$. This gives the Lie algebra $L(A, \circ)$  which is isomorphic to $L(A)$, because the addition and multiplication   constructed using  the inverse of the BCH formula  only depends on the group multiplication.
\item Therefore the formula 
 \[a\cdot b=-\sum_{i=0}^{p-2}{\frac 1{2^{i}}} ({2^{i}}a)*b\] from Theorem \ref{5}  applied to the multiplicative group of some strongly nilpotent brace $A$  would give a pre-Lie algebra $A$ whose Lie algebra $L(A)$  is isomorphic to the Lie algebra $L(A, \circ )$ obtained from the inverse of the BCH formula  applied to the multiplicative group $(A, \circ )$ of brace $A$ (because the formula from Theorem \ref{5} reverses the formula which gives the group of flows). 
\item Therefore, for multiplicative groups of strongly nilpotent braces, this may be useful for calculations related to the inverse of the BCH formula.
\end{itemize}
\section{ Applications: braces of cardinality $p^{4}$}
 In this section we will describe all nilpotent pre-Lie algebras of cardinality $p^{4}$ generated by one element and then use Theorem \ref{5} to describe all strongly nilpotent ${\mathbb F}_{p}$-braces of cardinality $p^{4}$ generated by one element.

 Let $A$ be a brace with the usual operations $+, \circ , *$ where $a\circ b=a*b+a+b$. Recall that $A^{[1]}=A$ and $A^{[i+1]}=\sum_{j=1}^{i} A^{[j]}*A^{[i+1-j]}$.  Notice that $A^{[i+1]}\subseteq A^{[i]}$ since $A^{[2]}\subseteq A$. We begin with the following lemma.

\begin{lemma}\label{nil}
 Let $A$ be a strongly nilpotent brace of cardinality $p^{4}$ for some prime number $p$.  Let $k$ be a natural number and suppose that $A^{[k]}=0$, then $A^{[6]} =0$.
\end{lemma}
\begin{proof} 
 Consider sets $A$, $A^{[2]}$, $A^{[3]}$, $A^{[4]}$. Then we have the following two cases: 
\begin{itemize}
\item $A\neq A^{[2]}\neq A^{[3]}\neq A^{[4]}$. Then $A^{[i]}/A^{[i+1]}$ have cardinality $p$ for $i=1,2,3,4$, since the cardinality $A$ is $p^{4}$ so it could not be bigger (it could not be smaller, since every subgroup of the additive group is a $p$-group). In this case we see that the cardinality of $A^{2}$ is $p^{3}$ and the cardinality of $A^{[3]}$ is $p^{2}$ and the cardinality of $A^{[4]}$ is $p$.

 Suppose that $A^{[6]}\neq 0$  then $A^{[4]}=A^{[6]}$ and  $A^{[5]}\subseteq A^{[6]}$. Let $a\in A^{[i]}, b\in A^{[j]}, c\in A^{[l]}$ for some $i,j,l$, then \[(a*b)*c-a*(b*c)-(b*a)*c+b*(a*c)\in A^{[i+j+l+1]}\] by \cite{Iyudu}, so \[A^{[3]}*A^{[3]}\subseteq A*A^{[4]}+A^{[4]}*A+A^{[7]}\subseteq A^{[6]}*A+A*A^{[6]}+A^{[7]}\] (when applied for $i=3,j=2,l=1$ and for $i=3, j=1, l=2$). 

 It follows that  \[A^{[6]}= A^{[5]}* A+A ^{[4]}* A^{[2]}+A^{[3]}*A^{[3]}+A^{[2]}* A^{[4]}+A* A^{[5]}\subseteq A^{[7]}.\] Continuing in this way we get that $A^{[4]}\subseteq A^{[6]}=0$, a contradiction.
\item  $A^{[i]}=A^{[i+1]}$ for some $i\in \{1,2,3\}$. We will show that $A^{[3]}=0$ in this case.  Notice that $A\subseteq  A^{2}$ implies $A^{[j]}\subseteq A^{[j+1]}$ for every $j$, consequently $A=A^{[k]}=0$, a contradiction. Similarly $A^{[2]}\subseteq A^{[3]}$ implies   
 $A^{[j]}\subseteq A^{[j+1]}$ for every $j>1$, consequently $A^{[2]}=A^{[k]}=0$. 
 Consequently we only need to consider the case when $A^{[3]}=A^{[4]}$. Let $x,y,z,t\in A$. 
  Observe that since $A$ is a strongly nilpotent brace we can apply Lemma \ref{fajny} for 
 $a=x*y, b=z, c=t$ and then for $a=z, b=x*y, c=t$ and subtracting  we get that 
\[((x*y)*z)*t-(x*y)*(z*t)-(z*(x*y))*t+z*((x*y)*t)\in A^{[5]}\] (this also follows from the fact that associated graded structures of braces are pre-Lie algebras \cite{Iyudu}). 
  Therefore \[(x*y)*(z*t)\subseteq A^{[3]}*A+A*A^{[3]}\subseteq A^{[5]}.\]
 Therefore we get that $A^{[3]}\subseteq A^{[4]}$ implies $A^{[4]}\subseteq A^{[5]}$. Now we can use this to show by induction that $A^{[j]}\subseteq A^{[j+1]}$ for all $j>3$, consequently \[A^{[3]}\subseteq A^{[4]}\subseteq \ldots \subseteq A^{[k]}=0,\] therefore $A^{[5]}=0$, as required. 
\end{itemize}
\end{proof}

Let $A$ be a pre-Lie algebra, and by $A^{i}$ we denote the linear space over $\mathbb F_{p}$ spanned by all products of $i$ or more elements from $A$. Observe that $A^{1}=A$ and $A^{i+1}=\sum_{j=1}^{i}A^{j}\cdot A^{i+1-j}$.

Our next result is as follows.
 
\begin{lemma}\label{Aggi} Let $p$ be a prime number larger than $2^{6}$ and let ${\mathbb F}_{p}$ be the field consisting of $p$ elements. 
 Let $(A, +, \circ )$ be a  right nilpotent $\mathbb F_{p}$-brace of cardinality $p^{4}$. Then
 $(A, \circ )$ is the group of flows of some  nilpotent pre-Lie algebra $(A, +, \cdot )$.   Moreover the product of any $6$ elements in this  pre-Lie algebra $A$  is zero.
\end{lemma}
\begin{proof}  By  a result by Rump, every brace of order $p^{n}$ is left nilpotent \cite{Rump}. By a result from \cite{Engel} a brace which is right nilpotent and left nilpotent is strongly nilpotent, therefore our brace $A$ is strongly nilpotent. By Lemma \ref{nil} $A$ has nilpotency index $6$ or less.  The result now  follows from Theorem \ref{5}.
\end{proof}

\begin{lemma} Let assumptions and notation be as in Lemma \ref{Aggi}. Suppose that the  $A$ is generated as a  brace  by one element $x$. Then $(A, \circ )$ is the group of flows of a nilpotent  pre-Lie algebra generated by one element $x$. 
\end{lemma}
\begin{proof} 
 Let $A$ be the pre-Lie algebra obtained as in Proposition \ref{5} from brace $A$. Then by Theorem \ref{5} $(A, \circ )$ is the group of flows of  the  pre-Lie algebra $(A, +, \cdot )$. By the definition of the group of flows elements
 \[x\cdot x, x\cdot(x\cdot x),  (x\cdot x)\cdot x, (x\cdot x)\cdot (x\cdot x), x\cdot (x\cdot (x\cdot x)),\ldots \] can be presented as linear combination of products  
 \[x, x*x, x*(x*x), (x*x)*x, (x*x)*(x*x), x*(x*(x*x)), \ldots .\]  

Since $A^{[6]}=0$ by Lemma \ref{Aggi}  these elements span brace $(A, +, \circ )$ and hence span  pre-Lie algebra $A$ as a linear space over ${\mathbb F}_{p}$. Notice that elements 
  \[x, x*x, x*(x*x), (x*x)*x, (x*x)*(x*x), x*(x*(x*x)), \ldots \] are linear combination of elements 
 \[x\cdot x, x\cdot(x\cdot x),  (x\cdot x)\cdot x, (x\cdot x)\cdot (x\cdot x), x\cdot (x\cdot (x\cdot x)),\ldots \] (by the proof of Theorem \ref{5}). 
 So the pre-Lie algebra $(A, +, \cdot )$ is generated as pre-Lie algebra by element $x$ and $A^{6}=0$.

\end{proof}
 
\begin{lemma}\label{sweetest}
 Let $(A, +, \cdot )$ be a pre-Lie algebra over $\mathbb F_{p}$ generated as a pre-Lie 
algebra by one element $x$. Suppose that $A$ has cardinality $p^{4}$ and $A^{6}=0$ and $A^{4}\neq 0$. Then $A^{3}\cdot A^{2}=0$, $A^{4}\cdot A=0$, $A\cdot A^{4}=0$.
 Moreover, $x^{2}\cdot (x\cdot x^{2})=0$.
\end{lemma}
\begin{proof} Reasoning as in Lemma \ref{nil} we get that if $A^{5}\neq 0$ then $A^{4}=A^{5}$, and $A^{6}=0$. 
  Observe that we have the following pre-Lie algebra relations:

Let $a\in A^{3}$, then  
\[(x\cdot a)\cdot x-x\cdot (a\cdot x)=(a\cdot x)\cdot x-a\cdot x^{2}.\]
 Notice that $A^{4}=A^{5}$ implies $A\cdot A^{4}\subseteq A\cdot A^{5}\subseteq A^{6}=0$, similarly $A^{4}\cdot A\subseteq A^{5}\cdot A\subseteq A^{6}=0$.  This along with the above relation imply $A^{3}\cdot A^{2}=0$.
 We  also have the following relation: 
\[(x\cdot x^{2})\cdot x^{2}-x\cdot (x^{2}\cdot x^{2})=(x^{2}\cdot x)\cdot x^{2}-x^{2}\cdot (x\cdot x^{2}),\]
  which implies
\[x^{2}\cdot (x\cdot x^{2})\in A^{3}\cdot A^{2}+A\cdot A^{4}\subseteq  A^{6}=0.\]
\end{proof}

  Let $(A, +, \cdot )$ be a pre-Lie algebra over ${\mathbb F}_{p}$.  For $\alpha \in {\mathbb F}_{p}$ and $a\in A$, by $\alpha a$ we will  denote  the sum of $\alpha $ elements $a$, and we denote  $a^{2}=a\cdot a$.
\begin{proposition}\label{aloha}
 Let $A$ be a pre-Lie algebra over $\mathbb F_{p}$ generated as a pre-Lie algebra by one element $x$. Suppose that $A$ has cardinality $p^{4}$ and $A^{5}\neq 0$ and  $A^{6}=0$. 
  Then the following holds:
\begin{itemize}
\item Elements $x, x^{2}, x^{2}\cdot x, x^{2}\cdot (x^{2}\cdot x)$ form a  base of the pre-Lie $A$ as a vector space over ${\mathbb F}_{p}.$
\item All products of $5$ or more elements $x$ are zero, except of the element $x^{2}\cdot (x^{2}\cdot x)\neq 0$.  
\item The following relations hold,
\[x^{2}\cdot x^{2}=(x^{2}\cdot x)\cdot x+x\cdot (x^{2}\cdot x),\]
 \[(x\cdot x^{2})\cdot x=x\cdot (x\cdot x^{2})=0.\]
 \item For some $\alpha, \beta , \gamma \in {\mathbb F}_{p}$ the following relations hold: 
  \[ (x^{2}\cdot x)\cdot x=\beta (x^{2}\cdot (x^{2}\cdot x)), x\cdot(x^{2}\cdot x)=\gamma  ( x^{2}\cdot (x^{2}\cdot x)),\]
\[x\cdot x^{2}=\alpha (x^{2}\cdot (x^{2}\cdot x)).\]
\end{itemize} 
Therefore, every element from  $A^{4}$ equals element $x^{2}\cdot(x^{2}\cdot x)$ multiplied by some element from $\mathbb F_{p}$.
 Notice that the above relations give a well defined pre-Lie algebra.
\end{proposition}
\begin{proof} By Lemma \ref{sweetest}  the only possible non-zero product of  five or more elements $x$ is $x^{2}\cdot (x^{2}\cdot x)$.
 Reasoning as in Lemma \ref{nil} we obtain that $A^{3}/A^{4}$ has dimension $1$ and $A^{4}=A^{5}$ also has dimension $1$ as a vector space over field  $\mathbb F_{p}$.
 Therefore, $A^{4}=A^{5}=\mathbb F_{p}(x^{2}\cdot(x^{2}\cdot x)).$ Therefore, \[(x^{2}\cdot x)\cdot x=\beta (x^{2}\cdot (x^{2}\cdot x)), x\cdot(x^{2}\cdot x)=\gamma  ( x^{2}\cdot (x^{2}\cdot x)),\] for some $\beta , \gamma  \in \mathbb F_{p}$. 

 Notice that $A^{2}$ is not a subset of $A^{3}=A^{2}\cdot A+A\cdot A^{2}$, since we could then substitute $A^{3}$ instead of $A^{2}$ on the right hand side several times and obtain that $A^{2}\subseteq A^{3}\subseteq A^{4}\subseteq A^{5}\subseteq A^{6}=0$.  Similarly $A$ is not a subset of $A^{2}$.

 Because $A^{3}/A^{4}$ has dimension $1$, then either $x^{2}\cdot x-\alpha x\cdot x^{2}\in A^{4}$ for some $\alpha \in \mathbb F_{p}$ or $x\cdot x^{2}\in A^{4}$. If $x^{2}\cdot x-\alpha x\cdot x^{2}\in A^{4}$ then 
$x^{2}\cdot(x^{2}\cdot x)-\alpha x^{2}\cdot (x\cdot x^{2})\in A^{6}$, and since $x^{2}\cdot (x\cdot x^{2})=0$ by Lemma \ref{sweetest} then we would get $A^{2}\cdot A^{3}=0$.  This and Lemma \ref{sweetest}  would imply $A^{5}=0$. So, since $A^{5}\neq 0$ then $x\cdot x^{2}\in A^{4}=A^{5}$.
Notice   that $x\cdot (x\cdot x^{2})\in A^{6}=0$ and $(x\cdot x^{2})\cdot x\in A^{6}=0$ since $x\cdot x^{2}\in A^{5}$.

 To obtain relation $x^{2}\cdot x^{2}=(x^{2}\cdot x)\cdot x+x\cdot (x^{2}\cdot x) $ we can use the pre-Lie algebra relation
\[(x^{2}\cdot x)\cdot x -x^{2}\cdot x^{2}=(x\cdot x^{2})\cdot x-x\cdot (x^{2}\cdot x)\] and notice that 
$x\cdot x^{2}\in A^{4}=A^{5}$ and so $(x\cdot x^{2})\cdot x\in A^{6}=0.$ 

 This implies the relations from our proposition.
 
 To see that the relations assumed in our theorem give a well defined pre-Lie algebra observe that every element from $A^{i}$ can be writen as sums of elements from the base $x, x^{2}, x^{2}\cdot x, x^{2}\cdot (x^{2}\cdot x)$ which are also from $A^{i}$ so the degree will stay the same or increase (by the degree we mean the number of occurence of $x$ in any product).

 We will check that the algebra is well defined by using the multiplication table, by considering all products $(a\cdot b)\cdot c$ and $a\cdot (b\cdot c)$ of elements from our base, and use the multiplication table to substitute sums of elements from the base for each product 
$a\cdot b$ and $b\cdot c$, and then use this to calculate $(a\cdot b)\cdot c$ and $a\cdot (b\cdot c)$. 
 
 Because any  product of  $6$ or more elements $x$ will be zero, and by substituting elements from our base we cannot decrease the degree (the number of appearance of $x$ in each product), then  we need to only consider products $a\cdot (b\cdot c)$ and $(a\cdot b)\cdot c$ where $a,b,c$ are elements from our base, and $x$ appears at most $5$ times in each product  $a\cdot (b\cdot c)$.

Therefore it  is easy to check with the multiplication table that all of the pre-Lie algebra relations 
 \[(a\cdot b)\cdot c-a\cdot(b \cdot c)=(b\cdot a)\cdot c-b\cdot (a\cdot c),\] are satisfied as we only need to consider the case when $a=x,b=x\cdot x$ and $c\in \{x, x\cdot x\}$ and the case when $a=x^{2}\cdot x$, $b=x, c=x$. 
 Therefore,  every structure of this type is a well defined pre-Lie algebra.
\end{proof}

\begin{proposition}\label{alloha}
 Let $A$ be a pre-Lie algebra over $\mathbb F_{p}$ generated as a pre-Lie algebra by one element $x$. Suppose that $A$ has cardinality $p^{4}$ and $A^{4}\neq 0$ and  $A^{5}=0$. Then 
  the following holds:
\begin{itemize}
\item Elements $x, x^{2}, a, b$ form a  base of the pre-Lie $A$ as a vector space over ${\mathbb F}_{p}$ for some $a\in \{x\cdot x^{2}, x^{2}\cdot x\}$ and some $b\in A^{4}, b\notin A^{3}$.
 Moreover, every element from $A^{i}$ will be a sum of elements from this basis which belong to $A^{i}$ for each $i$.
\item All products of $5$ or more elements from $A$ are zero.  
\item $A^{2}/A^{3}$, $A^{3}/A^{4}$ and $A^{4}$ are one-dimensional vector spaces over ${\mathbb F}_{p}$.
\item The following relation holds:
\[ (x^{2}\cdot x)\cdot x-(x\cdot x^{2})\cdot x =x^{2}\cdot x^{2}-x\cdot (x^{2}\cdot x).\] 
\item For some $\alpha , \beta \in {\mathbb F}_{p}$, not both zero, we have \[\alpha (x\cdot x^{2})+\beta (x^{2}\cdot x)\in A^{4},\]
 and consequently the following relations hold in $A^{4}$:
\[ \alpha x\cdot (x\cdot x^{2})+\beta x\cdot (x^{2}\cdot x)=0  \]
\[\alpha (x\cdot x^{2})\cdot x+\beta (x^{2}\cdot x)\cdot x=0   \]
\end{itemize} 
 Notice that  the above relations give a well defined pre-Lie algebra.
\end{proposition}
\begin{proof} We can use a similar proof as in Proposition \ref{aloha}. We can take any non-zero product of some copies of element $x$ from $A^{4}$ to be element $b\neq 0$. Notice that either  $b=a\cdot x$ or $b=x\cdot a$ for some $a\in A^{3}$, and  we can add this $a$ to the base (notice that if all elements from $A^{2}\cdot x$ and $x\cdot A^{2}$ are zero then $x^{2}\cdot x^{2}=0$).  Reasoning similarly as in Lemma \ref{nil} we get that   
 $A^{2}/A^{3}$, $A^{3}/A^{4}$ and $A^{4}$ are one-dimensional vector spaces over ${\mathbb F}_{p}$.

Notice that $A^{3}/A^{4}$ has dimension $1$ as a vector space over ${\mathbb F}_{p}$, which gives $\alpha , \beta $.

 The fact that this pre-Lie algebra is well defined follows from the fact that   every element from $A^{i}$ will be a sum of elements from thie basis which belong to $A^{i}$ for each $i$. Therefore, by substituting elements from basis for $a\cdot b$ in a product $(a\cdot b)\cdot c$ we cannot decrease the degree of this product. Therefore, reasoning similarly as in Proposition \ref{aloha} we only need to check that the pre-Lie algebra relations
  \[(a\cdot b)\cdot c-a\cdot(b \cdot c)=(b\cdot a)\cdot c-b\cdot (a\cdot c),\] are satisfied for $a=x, b=x^{2}, c=x$. Notice that this relation holds by assumptions.  
\end{proof}

\begin{proposition}\label{braces} Let $A$ be a pre-Lie algebra over field ${\mathbb F}_{p}$ such that $A^{4}=0$ and $A$ is generated by 
 element $x\in A$ as a pre-Lie algebra. Suppose that $A$ has cardinality $p^{4}$.
   Then elements $x, x^{2}, x^{2}\cdot x, x\cdot x^{2}$ form a base of $A$ as a vector space over ${\mathbb F}_{p}$, 
moreover all products of more than $3$ elements from $A$  are zero and there are no other relations in this pre-Lie algebra $A$.
Notice that this  gives a  well defined pre-Lie algebra.
\end{proposition}
\begin{proof} Notice that elements $x$ and $x^{2}$ appear in every product of monomials so $x$ and $x^{2}$ cannot be sums of products of  more than $2$ elements  $x$, as by substituting such relations in every product of the right hand side several times we would get $x, x^{2}\in A^{[6]}=0$. Therefore, if  elements $x\cdot x^{2}$ and $x^{2}\cdot x$ are linearly independent then 
  $x,x^{2}, x\cdot x^{2}, x^{2}\cdot x$ will be a basis of $A$ as a linear space over $\mathbb F_{p}$. 
 Notice that  we cannot have any relation involving $x\cdot x^{2}$ and $x^{2}\cdot x$, because then our pre-Lie algebra would have dimension less than $3$ over ${\mathbb F}_{p}.$ By constructing a multiplication table, as in the proof of Proposition \ref{aloha}, we see that our pre-Lie algebra is well defined. 
\end{proof}
 The results in this section yield the following corollary: 
\begin{corollary}  
 Let $p$ be a prime number larger than $2^{6}$ and let ${\mathbb F}_{p}$ be the field consisting of $p$ elements. 
 Let $(A, +, \circ )$ be a  right nilpotent $\mathbb F_{p}$-brace of cardinality $p^{4}$. Then $A$ is the group of flows of one of pre-Lie algebras from Propositions \ref{aloha}, \ref{alloha} and \ref{braces}. 
\end{corollary}
\section{ Left nilpotent radical}

 In associative algebra sums of nilpotent ideals are nilpotent ideals. It was shown in \cite{CKM}  that a sum of a finite number of  left 
nilpotent ideals in a pre-Lie algebra is a left nilpotent ideal. 
  In this section we obtain an analogon of this result for braces. Recall that an ideal $I$ in a brace $(A, +, \circ )$ is left  nilpotent if $I^{n}=0$ for some $n$ where $I^{1}=I$ and $I^{i+1}=I*I^{i}$.  
 We will first prove two supporting lemmas:

\begin{lemma}\label{111}
 Let $(A, +, \circ )$ be a left brace and let $I, J$ be left nilpotent ideals  in $A$. Denote  $I+J=\{a+b:a\in I, b\in J\}$. Then  for every $c \in A$ we have 
\[(I+J)*c\subseteq I*c+J*c+I*(J*c).\]
 \end{lemma}
\begin{proof}  Let $\lambda _{x}(y)=x*y+y$, then it is known that \[\lambda _{x\circ z}(y)=\lambda _{x}(\lambda _{z}(y))\] for all $x,y,z\in A$. Let $a^{-1}$ be the inverse of $a$ in the group $(A, \circ )$.
  Observe that for any $c\in A$ we  have \[(I+J)*c\subseteq I*c+J*c+I*(J*c).\] Indeed, let $a\in I, b\in J$ then,  
\[(a+b)*c=(a+\lambda _{a}(\lambda _{a^{-1}}(b))*c=(a+b'+a*b')*c=a*c+b'*c+a*(b'*c)\] where $b'=\lambda _{a^{-1}}(b)\in J$ since $b\in J$ and $J$ is an ideal in $A$.
\end{proof}

\begin{lemma}\label{112}
 Let $(A, +, \circ )$ be a left brace and let $I, J$ be left nilpotent ideals  in $A$. Denote,  $I+J=\{a+b:a\in I, b\in J\}$. For $c\in A$ denote $G_{I,J, c}=I*(J*(I*c))+ I*(J*(I*(-c))).$ Then  for every $c \in A$ we have 
  \[I*(J*(c))\subseteq I*c+I*(-c)+I*(I*c)+I*(I*(-c))+J*(I*c)+J*(I*(-c))+G_{I, J, c}.\]
 \end{lemma}
\begin{proof} It suffices to show that 
  \[I*(J*(c))\subseteq J*(I*c)+ I*c+ I*(-c) +(I+J)*(I*c)+(I+J)*(I*(-c))\]
 and then apply Lemma \ref{111} to get:
\[(I+J)*(I*c)\subseteq I*(I*c)+J*(I*c)+I*(J*(I*c)).\] 
 It remains to show that 
  \[I*(J*(c))\subseteq J*(I*c)+(I+J)*(I*c)+ I*c+(I+J)*(I*(-c))+I*(-c).\]
 Let $a\in I, b\in J$, observe that 
\[(a+b+a*b)*c=a*c+b*c+a*(b*c)   \]
\[ (a+b+a*b)*c=((a+b)+\lambda _{a+b}(a'))*c=(a+b)*c+ a'*c+(a+b)*(a'*c) , \] where $a'=\lambda _{(a+b)^{-1}}(a*b)\in I\cap J$, since $a*b\in I\cap J$. Therefore, 
\[(a+b)*c-a*c-b*c=a*(b*c)-a'*c-(a+b)*(a'*c).\]
 By applying it to $a:=b$ and $b:=a$ we get 
\[(b+a)*c-b*c-a*c=b*(a*c)-a''*c-(b+a)*(a''*c),\]
 for some $a''\in J\cap I$.
 By comparing the above equations we obtain
 \[a*(b*c)=b*(a*c)+a'*c+(a+b)*(a'*c)-a''*c-(b+a)*(a''*c).\]
Since $a*(-b)=-(a*b)$ in every brace and $a,b,c$ were arbitrary this implies 
  \[I*(J*(c))\subseteq J*(I*c)+ I*c+ I*(-c) +(I+J)*(I*c)+(I+J)*(I*(-c)).\]
\end{proof} 

\begin{theorem}
 Let $(A, +, \circ )$ be a left brace and let $I, J$ be left nilpotent ideals  in $A$. Then $I+J=\{a+b:a\in I, b\in J\}$ is a left nilpotent ideal in $A$.
\end{theorem}
\begin{proof} Notice first  that  $I+J$ is an ideal in $A$ by  Lemma 3.7 \cite{ksv}.
Denote $ E_{n}(I, J)$ to be  the sum of elements which belong to some sets \[P_{n}*(P_{n-1}*(\cdots *(P_{1}*c)))\] and some sets \[P_{n}*(P_{n-1}*(\cdots *(P_{1}*(-c))))\]  where each $P_{i}\in \{I, J\}$, $c\in A$. By Lemma \ref{111} we get that  
  \[E_{n}(I+J, I+J)\subseteq \sum_{n\leq k\leq 2n}E_{k}(I,J)\] for each $n$ (it can be proved by induction on $n$). 

 To show that $I+J$ is nilpotent it suffices to show that for a sufficiently large $k$, all $E_{k}(I,J)=0. $ 
 Let $\alpha $  be such that $I^{\alpha }$, $J^{\alpha }=0$. 
 
 Let $T_{k,i,n}$ consist of elements from sets \[P_{n}*(P_{n-1}*(\cdots *(P_{1}*c)))\] and sets \[P_{n}*(P_{n-1}*(\cdots *(P_{1}*(-c))))\] for which \[P_{1}=\ldots =P_{k}=I\]
and \[ P_{k+1}=\ldots =P_{k+i}=J\]  and $P_{k+i+1}=I$. Notice that for $k=0$ we have  $P_{1}=J$. 

 Observe that $T_{k, \alpha +i, n}=0$ for every $i, k, n\geq 1$ because 
\[P_{k+\alpha +1}*(P_{k+\alpha }*(\ldots *(P_{1}*c)))=P_{k+\alpha +1}*(P_{k+\alpha }*(\ldots *(P_{k+2}*D))) =0\] for
 \[D=P_{k +1}*(P_{k }*(\ldots *(P_{1}*c)))\in J\] since 
$ J^{\alpha} =0$ and \[P_{k+\alpha +1}=P_{k+\alpha }=\ldots =P_{k+1}=J.\] Similarly $T_{\alpha +k, i, n}=0$ for every $i, n,k\geq 1$.
  
 We will use a similar argument as in \cite{CKM}.Suppose that $i>1$. Notice that Lemma \ref{112} applied to ideals $P_{k+i+1}=I$ and $P_{k+i}=J$ and to \[C=P_{k+i-1}*(P_{k+i-2}*(\cdots *(P_{1}*c)))\] yields  $P_{k+i+1}*(P_{k+i}*C)=I*(J*C)$
 and for $i>1$ we have 
 \[I*(J*C)\subseteq I*C+I*(-C)+I*(I*C)+J*(I*C)+I*(I*(-C))+J*(I*(-C))+G_{I, J, C}.\]  
where \[G_{I,J, C}=I*(J*(I*C))+I*(J*(I*(-C))).\] Therefore, $I*(J*C) \subseteq T_{k, i-1}$, and hence \[T_{k, i,n}\subseteq T_{k, i-1,n-1}+T_{k, i-1,n}+T_{k, i-1,n+1}\] for $i>1$. 
 If $i=1$ then we obtain 
\[T_{k,1, n}\subseteq  \sum _{j\leq \alpha , m\in {\{n-1,n,n+1\}}}T_{k+1, j, m}+T_{k+2, j, m}+\cdots .\] Applying it several times (at most $(\alpha +1) ^{2}$ times) we can obtain $k>\alpha $,  so we  eventually obtain zero, since $T_{k, i, n}=0$ for $k>\alpha $, $n>2(\alpha +1) ^{2}$.
\end{proof} 
Notice that every left nilpotent ideal is a solvable ideal, and that a sum of a finite number of left  nilpotent ideals in a brace is a left nilpotent ideal. Therefore every finite brace contains the largest nilpotent ideal, and this ideal is also contained 
 in the largest solvable ideal of this brace (it is known that a sum of two solvable ideals in a brace is a solvable ideal \cite{ksv}). This mirrors the situation for pre-Lie algebras from \cite{CKM}. We obtain the following. 

\begin{corollary}
  If $A$ is a finite brace, then $A$ contains the largest left nilpotent ideal, which is the sum of all left nilpotent ideals in $A$. We will call this ideal the left nilpotent radical of $A$.
\end{corollary}

In  \cite{ksv}, the Wedderburn radical of a brace was defined as a sum of all ideals in $A$ which are both left nilpotent and right nilpotent. 
 In Lemma 6.4 \cite{ksv}  it was shown that the Wedderburn radical in any brace $A$ is solvable.  We get the following result.

\begin{corollary}
  If $A$ is a finite brace, then the Wedderburn radical of $A$ is left nilpotent.
\end{corollary}
 
 This suggests the following (open) questions.
 \begin{question}
 Let $A$ be a finite brace. Is the Wedderburn radical of $A$ strongly nilpotent?
\end{question}
 \begin{question}
 Let $A$ be a left brace, and $I, J$ be two strongly nilpotent ideals in $A$. Is $I+J$ a strongly nilpotent ideal in $R$?
\end{question}
 \begin{question}
 Let $A$ be a left brace, and $I, J$ be two right nilpotent ideals in $A$. Is $I+J$ a right nilpotent ideal in $R$?
\end{question}
 The above questions have some similarity to the   Koethe conjecture in ring theory,  which states that a sum of two nil  right ideals in a ring is  nil.\\

Some interesting results on nilpotent semi-braces were recently obtained in \cite{Cedo} and it is an open problem whether this result  can be generalised to semi-braces.

$ $

{\bf Acknowledgments.} The author is grateful to Michael West and  Alicja Smoktunowicz  for their help in preparing the  Introduction.  This research was supported by  EPSRC grants 
EP/V008129/1 and EP/R034826/1. The author is also very grateful to the referee for helpful comments.


\begin{thebibliography}{99}
\bibitem{AG} A. Agrachev, R. Gamkrelidze, {\em Chronological algebras and nonstationary vector fields},
J. Sov. Math. 17 No1 (1981), 1650–1675.
\bibitem{A3}  S. A. Amitsur, {\em A general theory of radicals. III. Applications}, Am. J. Math. 76 (1954), 126--36. 
\bibitem{DB} D. Bachiller, {\em Counterexample to a conjecture about braces}, J. Algebra, 
Volume 453, (2016),  160--176.
\bibitem{Ba} D. Bachiller, {\em Classification of braces of order $p^{3}$}, J.  Pure  Appl. Algebra
Volume 219, Issue 8, (2015), Pages 3568--3603.
\bibitem{djc}  D. Bachiller, F. Ced{\'o}, E. Jespers, and J. Okni{\' n}ski, {\em Asymmetric product of left braces and simplicity; new solutions of the Yang-Baxter equation},  Commun. Contemp. Math., Vol. 21, No. 08, 1850042 (2019). 
\bibitem{Bai} C. Bai,  An Introduction to Pre-Lie Algebras, Wiley Online Library, 19 March 2021, https://doi.org/10.1002/9781119818175.ch7. 
 % \bibitem{ball} H. Meng, A. Ballester-Bolinches, R. Esteban-Romero, and N. Fuster-Corral, {\em On finite involutive Yang--Baxter groups},  Proc. Amer. Math. Soc., 149(2), (2021) %793--804.
\bibitem{Bokut} L. Bokut, Yu. Chen, {\em  Groebner-Shirsov Bases theory and Shirshov algorithm}, e-book, Novosybirsk, 2013.
\bibitem{TB} T.  Brzezi{\' n}ski, {\em     Trusses: Between braces and rings}, Trans. Am. Math. Soc, 372(6), (2018) 4149--4176. 
\bibitem{Burde}  D. Burde, {\em Left-symmetric algebras, or pre-Lie algebras in geometry and physics}, Cent. Eur. J. Math. 4 (2006) 323--357. 
\bibitem{Cedo} F. Catino, F. Ced{\' o}, P. Stefanelli, {\em Nilpotency in left semi-braces}, arXiv:2010.04939v1 [math.QA].
\bibitem{Catino} F. Catino, R. Rizzo, {\em  Regular subgroups of the affine group and radical circle algebras},
Bull. Aust. Math. Soc. 79 (2009), no. 1, 103--107.
\bibitem{cjo} F. Ced{\' o}, E. Jespers, J. Okni{\' n}ski, {\em Braces and the Yang-Baxter equation}, Comm. Math. Phys. 327, 
 (2014), 101--116.
\bibitem{Rio}  F. Ced{\' o}, E. Jespers, A. del Rio, {\em Involutive Yang-Baxter groups}, Trans. Amer. Math. Soc. 362 (2010), 2541--2558. 
\bibitem{cedo}  F. Ced{\' o}, E. Jespers, J. Okni{\' n}ski, {\em Constructing finite simple solutions of the Yang-Baxter equation}, Adv. Math. 391(3), (2021), 107968.
\bibitem{CKM} K. Soo  Cheng, H. Kim, H. Chul Myung, {\em On radicals of left symmetric algebras}, Comm. Algebra, 27 (7), (2007),  3161--3175.
\bibitem{LC} L. N. Childs, {\em Skew braces and the Galois correspondence for Hopf-Galois structures,} J. Algebra 511 (2018) 270--291. 
 \bibitem{L} L. N. Childs, {\em Elementary abelian Hopf Galois structures and polynomial formal groups}, J. Algebra 283 (2005) 292--316.

\bibitem{FC} F. Chouraqui, {\em  Garside Groups and Yang-Baxter Equation}, Comm. Algebra, 
Vol. 38, Issue 12, (2009), 4441--4460.
\bibitem{pent} I. Colazzo, E. Jespers, {\L}. Kubat, {\em Set-theoretic solutions of the Pentagon Equation},  Comm. Math. Phys. 380(4), (2020),1--22.
\bibitem{TC} T. Crespo, {\em Hopf--Galois structures on field extensions of degree twice an odd prime square and their associated skew left braces},  J. Algebra 565 (2021), 282--308. 
\bibitem{Dietzel} C. Dietzel,  {\em Braces of order $p^{2}q$},  J. Algebra Appl. Vol. 20, No. 08, 2150140 (2021).
\bibitem{doikou} A. Doikou, A. Smoktunowicz, {\em Set-theoretic Yang-Baxter and  reflection equations and quantum group symmetries.}  Lett. Math. Phys. 111, 105 (2021).
\bibitem{El} A. Elduque, H. Chul Mung, {\em On transitive left symmetric algebras}, Non-Associative Algebra and Its Applications, MAIA, volume 303, 1994, 114--121.
\bibitem{gateva} T. Gateva-Ivanova, {\em Set-theoretic solutions of the Yang-Baxter equation, Braces and Symmetric groups}, Adv. Math. 338 (2018) 649--701.
\bibitem{gv} L. Guareni, L. Vendramin, {\em  Skew braces and the Yang-Baxter equation}, Math. Comput. 86 (2017) 2519--2534.
\bibitem{Iyudu} N. Iyudu, {\em  Classification of contraction algebras and pre-Lie algebras associated to braces and trusses}, 
arXiv:2008.06033 [math.RA].
\bibitem{p} P. Jedli{\'  c}ka, A. Pilitowska, A. Zamojska-Dzienio, {\em The retraction relation for biracks}, Journal of Pure and Applied Algebra
Volume 223, Issue 8 (2019) 3594--3610.
\bibitem{JespersLeandro} E. Jespers, {\L}. Kubat, Arne Van Antwerpen,
L. Vendramin, {\em Radical and weight of skew braces and their applications to structure groups of solutions of the Yang-Baxter equation}, Adv. Math. 385 (2021),  Article 107767.
\bibitem{K}  E. I. Khukhro, {\em  p-automorphisms of finite p-groups}, London Mathematical Society Lecture Note Series, v.246, Cambridge University, Press, Cambridge, 1998.
\bibitem{kinneart} P. Kinnear, {\em The Wreath Product of Semiprime Skew Braces is Semiprime}, Comm. Algebra, 49 (2), (2021) 533--537.
\bibitem{ksv} A. Konovalov, A. Smoktunowicz, L. Vendramin, {\em On skew braces and their ideals}, Exp. Math., (2018)
Exp. Math., DOI: 10.1080/10586458.2018.1492476,, 95-104.
\bibitem{ILau} I. Lau, {\em An Associative Left Brace is a Ring}, J. Algebra Appl. Vol. 19, No. 09, 2050179 (2020).
\bibitem{LV} V. Lebed and L. Vendramin, {\em Cohomology and extensions of braces}, Pacific J. Math. 284 (2016), no. 1, 191--212.
\bibitem{M} D. Manchon,  {\em A short survey on pre-Lie algebras}, Noncommutative Geometry and Physics: Renormalisation, Motives, Index Theory, (2011), 89--102.

\bibitem{kayvan} K. Nejabati Zenouz, {\em Skew braces and Hopf-Galois structures of Heisenberg type}, J. Algebra 524 (2019) 187--225. 
\bibitem{rump} W. Rump, {\em Braces, radical rings, and the quantum Yang-Baxter equation}, J. Algebra
Volume 307,  (2007),  153--170.
\bibitem{Rump} W. Rump,  {\em The brace of a classical group}, 
Note Mat. 34 (2014) no. 1, 115--144.
\bibitem{RN} W. Rump, {\em Construction of finite braces}, Ann. Comb. 23, (2019), 391--416.
\bibitem{RumpEdi} W. Rump, {\em One-generator braces and indecomposable set-theoretic solutions to the Yang-Baxter equation},
  Proc. Edinb. Math. Soc. Volume 63,  (2020), pp. 676--696.
\bibitem{Agatka1} A. Smoktunowicz, A. Smoktunowicz,  {\em Set-theoretic solutions of the Yang-Baxter equation and new classes of  R-matrices},  Linear Algebra Appl.  546 (2018), 86--114. 
\bibitem{Engel} A. Smoktunowicz, {\em  On Engel groups, nilpotent groups, rings, braces and the Yang-Baxter equation},  Trans. Am. Math. Soc. 370(9), (2018)  6535--6564.
\bibitem{new} A. Smoktunowicz, {\em Algebraic approach to Rump's results on relations between braces and pre-Lie algebras},  J. Algebra Appl. online ready https://doi.org/10.1142/S0219498822500542.
\bibitem{SVB} A. Smoktunowicz, L. Vendramin, {\em On skew braces (with an appendix by N. Byott and L. Vendramin)},  J. Comb. Algebra, 2(1), 2018, 47--86.
\bibitem{Sysak} Y.P. Sysak, {\em Some examples of factorized gorups and their relation to ring theory},  Infinite Groups 1994. Proceedings of the International Conference, Ravello, Italy, May 23--27, (1994) 257--269. Walter de Gruyter, Berlin (1996).
\bibitem{V} L. Vendramin, {\em Problems on left skew braces},  Advances in Group Theory and Applications, 2019 AGTA 
 7 (2019), pp. 15--37.
\end{thebibliography}
\end{document}